\documentclass[12pt]{article}
\usepackage{amsmath}
\usepackage{amsfonts}
\usepackage{amssymb}
\usepackage{amsthm}
\usepackage{mathrsfs}
\begin{document}
\title{ Reproducing kernels of some weighted Bergman spaces }
\author{Guan-Tie Deng, Yun Huang, Tao Qian}

\author{Guan-Tie~Deng \thanks
{School of  Mathematical Sciences, Beijing
Normal University, Beijing, 100875. Email:denggt@bnu.edu.cn.\ This
work was partially supported by NSFC (Grant 11971042) and by SRFDP
(Grant 20100003110004)},\
Yun Huang \thanks{Corresponding author. Department of Mathematics, Faculty of Science and Technology, University of Macau, Macao (Via Hong
Kong). Email: yhuang12@126.com.},\
Tao~Qian \thanks{ Institute of Systems Engineering, Macau University of Science and Technology, Avenida Wai Long, Taipa, Macau). Email: tqian1958@gmail.com. The work is supported by the Macau Science and Technology foundation No.FDCT079/2016/A2 , FDCT0123/2018/A3, and the Multi-Year Research Grants of the University of Macau No. MYRG2018-00168-FST.}
}
\date{}
\maketitle
\begin{center}
\begin{minipage}{165mm}
{{\bf Abstract} \\
Herein, the theory of Bergman kernel is developed to the weighted case. A general form of weighted Bergman reproducing kernel is  obtained, by which we can calculate concrete Bergman kernel functions for specific weights and domains. }\\


\end{minipage}
\end{center}

\section{Introduction}
\newtheorem{theorem}{\sc{Theorem}}
\newtheorem{lemma}{\sc{Lemma}}
\newtheorem{corollary}{\sc{Corollary}}

The theory of Bergman spaces has, in the past several  decades, become important in complex analysis of both one and several complex variables, see \cite{SGK2} and \cite{SGK1}. Recall that, for an arbitrary domain $\Omega\subset\mathbb C^n$, the Bergman space $A^p(\Omega)$ is defined as the collection of analytic functions $F$ that satisfy
\begin{equation*}
\|F\|_{A^p}=\left\{\int_{\Omega}|F(z)|^pdA(z)\right\}^{\frac1p}<\infty,
\end{equation*}
where $dA(z)=dxdy$ is the Lebesgue measure on $\mathbb C^n$.
When $p=2$, the reproducing kernel plays an important role in the Hilbert space. For the classical Bergman spaces the reproducing properties and biholomorphic invariance are investigated in \cite{SM} and \cite{PD}.

Bergman kernels have also been considered in some weighted cases.
Since a reproducing kernel deliveries certain fundamental information of the corresponding space, it is important to obtain the concrete form of the kernel function.
However, we must confess that their weighted Bergman kernel can almost never be calculated explicitly except for some special cases. Among the latter, the Bergman kernels are given for functions defined on the Hermitian ball and polydisc in \cite{SGK2}. Some concrete expressions of the Bergman kernel are also available for some classical homogeneous bounded symmetric domains of Cartan in \cite{HLG}.

Since the set of polynomials is dense in the Bergman spaces of bounded domains, the corresponding reproducing kernels can be obtained from the general representation formula $K(z,w)=\sum_{n=1}^{\infty}\phi_n(z)\overline{\phi_n(w)}$, where $\{\phi_n\}$ is any complete orthonormal basis obtained through orthonormalization of polynomials.
However, for unbounded regions, density of polynomials can not be guaranteed. In such case appropriate weight functions are introduced to make the weighted polynomials dense in certain Bergman spaces of unbounded domains. Then the kernel representations can be similarly obtained. In this paper, we apply the Laplace transform to the case of Bergman spaces on tube domains, which will be an effective and new method to calculate reproducing kernels.

Herein, we develop the theory of the weighted Bergman spaces and obtain a general representation formula of the kernel function for the spaces on tubular domains. As a complementary party to the general study, we calculate the concrete forms of the Bergman kernels for some special weights on the tube domains.

In some previous studies the concerned reproducing kernels can be computed by using our general representation formula, since the given weight functions satisfy the set conditions in our theorems. For example, taking the weight function $\rho(iy)=\frac{2y^{2q-2}}{\pi\Gamma(2q-1)}$, we can derive that the Bergman--Selberg reproducing kernel on the upper half plane is in the form of $K_q(z,\omega)=\Gamma(2q)\left(\frac{i}{z-\bar\omega}\right)^{2q}$ with $q>\frac12$, which is introduced in \cite{SS}. For the weight function $\rho(iy)=y^{v-1}$, a direct computation gives that the corresponding kernel is of the form $K(z,w)=\frac{2^{v-1}v}{\pi}\left(\frac{z-\bar w}{i}\right)^{-v-1}$, which is studied in \cite{DBAB}.
For weighted Bergman spaces associated with Lorentz cones, referring to \cite{DBABMMPFR}, the kernel functions can be also obtained by our formula.
The related calculation process is given as an example in the final section.
Especially, in the un-weighted case, i.e., letting $\rho(iy)=1$, it follows that the classical Bergman kernel on the upper half plane is $K(z,w)=\frac{1}{\pi}\left(\frac{i}{z-\bar w}\right)^2$.

\section{Preliminaries}
Let $\Omega$ be an arbitrary domain (an open connected set) in the complex  $n$-dimensional Euclidean space $\mathbb C^n=\{z=x+iy:x,y\in\mathbb R^n\} $. Suppose that $\rho(z)$ is a positive continuous function on $\Omega$ that takes the value $0$ for $z\notin\Omega$. We consider the weighted volume measure
\begin{equation*}
dA_{\rho}(z)=\rho(z)dA(z),
\end{equation*}
where $dA(z)=dxdy $ is the Lebesgue measure on $\mathbb C^n$. For $p>0$, we denote by $L^p_\rho$ the space of measurable  functions on $\Omega $ such that
\begin{equation}
\|F\|_{L^p_{\rho}}=\left(\int_{\Omega}|F(z)|^pdA_{\rho}(z)\right)^{\frac1p}<\infty.\label{anorm}
\end{equation}
The space of such functions is called the weighted Lebesgue  space with weight $\rho$. The quantity $
\|F\|_{L^p_{\rho}} $ is called the norm of the function $F$, which is a true norm if $p\ge 1$.

We denote by $A^p_\rho$ the collection of functions $F$ that are holomorphic on $\Omega $ and satisfy the condition (\ref{anorm}). Such a class is called the weighted Bergman space with weight $\rho$. It is obvious that $A^p_\rho\subset L^p_\rho$.

We first assert that functions in the weighted Bergman space cannot grow too rapidly near the boundary.

\begin{lemma} Point-evaluation is a bounded linear functional in each weighted Bergman space
$A^p_\rho $. More specifically, each function $F\in A^p_\rho$
has the property
\begin{eqnarray}
|F(z)|&\leq& \left (\omega_{2n}\varepsilon_z\right )^{-\frac{1}{p}} \left (\delta_z\right)^{-\frac{2n}{p}}\|F\|_{A_{\rho}^p}. \label{riesz}
\end{eqnarray}
Here, $\omega_{2n}=\frac{\pi ^n}{n!} $ is the volume of unit ball $B_{2n}(0,1)$, $\delta_z=\min \{1,2^{-1}\mbox{dist} (z, \partial \Omega)\}$ where  $\mbox{dist} (z, \partial \Omega)\}$ is the distance from $z$ to the boundary of $\Omega$, and $\varepsilon_z=\min\{\rho(\zeta):|\zeta - z|\leq\delta_z\}$.
\end{lemma}
\begin{proof}
For fixed point $z\in \Omega$, the bounded closed ball $\overline{B_{2n}(z,\delta_z)}$ lies in $ \Omega$. Since $\rho(\zeta)$ is a positive continuous function on $\Omega$, then for any $\zeta\in \overline{B_{2n}(z,\delta_z)}$, we have $\varepsilon_z=\min\rho(\zeta)>0$. Therefore,
\begin{eqnarray*}
|F(z)|^p&\leq&\frac{1}{\omega_{2n}\delta_z^{2n}}\int_{B_{2n}(z,\delta_z)} |F(\zeta)|^pdA(\zeta)\\
&\leq&\frac{1}{\omega_{2n}\delta_z^{2n}}\int_{B_{2n}(z,\delta_z)} \frac{|F(\zeta)|^p}{\rho(\zeta)}dA_{\rho}(\zeta)\\
&\leq&\frac{1}{\varepsilon_z\omega_{2n}\delta_z^{2n}}\int_{B_{2n}(z,\delta_z)}|F(\zeta)|^p dA_{\rho}(\zeta)\\
&\leq&\frac{1}{\varepsilon_z\omega_{2n}\delta_z^{2n}}\int_{\Omega}|F(\zeta)|^p dA_{\rho}(\zeta)\\
&=&\frac{1}{\varepsilon_z \omega_{2n}\delta_z^{2n}}\|F\|^p_{A_{\rho}^p},
\end{eqnarray*}
which is the stated result.
\end{proof}
As a consequence of the above Lemma, we conclude that the weighted Bergman space $A^p_{\rho}$
is a Banach space when $1\le p <\infty$ and a complete metric space when $0<p<1$.
\begin{lemma}
Suppose that $\rho(z)$ is a positive continuous function on $\Omega$ that takes the value $0$ for $z\not\in \Omega $. For $0<p<\infty$, the weighted Bergman space
$A^p_\rho$ is closed in  $L^p_\rho $.
\end{lemma}
\begin{proof} Let $\{F_n\}$ be a sequence in $A^p_\rho$ and assume $F_n\rightarrow F$
in $L^p_\rho$. In particular, $\{F_n\}$ is a Cauchy sequence in $L^p_\rho$. Applying the previous Lemma, we see that $\{F_n\}$ converges uniformly on every compact subset of $\Omega $. Combining with the assumption that $F_n\rightarrow F$
in $L^p_\rho$, we conclude that $F_n\rightarrow F$
 uniformly on every compact subset of $\Omega $. Therefore,  $F$ is analytic in $\Omega$ and belongs to $A^p_\rho$.
\end{proof}

Now let $p=2$, $A^2_\rho$ is a Hilbert space with inner product
$$
\langle F,G\rangle_{\rho}=\int_{\Omega}F(z)\overline{G(z)}dA_{\rho}(z)
$$
 for $F,G\in A^2_\rho$.

Since each point evaluation functional $\mathcal T[F]=F(z)$ of $A^2_{\rho}$ is bounded,
  the Riesz representation theorem for Hilbert space guarantees existence of a unique function  $K(\zeta,z)=K_{ z}(\zeta)\in A^2_\rho$
  such that $F(z)=\langle F, K_{z}\rangle_{\rho}$ for every
  $F\in A^2_\rho$.  The function  $K(\zeta,z)$ is known as the reproducing kernel with weight $\rho$, or  the weighted  Bergman kernel function. It has the reproducing property
\begin{equation}
F(z)=\int_{\Omega }F(\zeta)\overline{K_{z}(\zeta)}dA_{\rho}(\zeta)\label{rieszrepresentationtheorem}
\end{equation}
for each function $ F \in A^2_\rho$. Taking  $F(z)=K(z,\zeta)$ for some $\zeta\in \Omega$,
we see that
\begin{equation}
K(z,\zeta)=\int_{\Omega }K(\eta,\zeta)\overline{K(\eta,z)}dA_{\rho}(\zeta)=\overline{K(\zeta,z)}.
\label{rieszrepresentationtheorem1}
\end{equation}
Thus the kernel function has the symmetry property $K(z,\zeta)=\overline{K(\zeta,z)}$, which also shows that $ K(z,\zeta) $ is analytic in $z$ and anti-analytic in $\zeta $.
Another consequence is the formula
\begin{equation}
K(z,z)=\int_{\Omega }|K(z,\zeta)|^2dA_{\rho}(\zeta)=\|K(z,\cdot)\|^2_{L^2_\rho}>0.
\label{rieszrepresentationtheorem2}
\end{equation}
In view of (\ref{rieszrepresentationtheorem2}), applying the Schwarz inequality to (\ref{rieszrepresentationtheorem}), there holds
$$
|F(z)|\leq \sqrt{K(z,z)}\|F\|_{A^2_\rho}.
$$
Then for each point $ \zeta \in \Omega $,
$$
\frac{1}{\sqrt{K(\zeta,\zeta)}}\leq \|F\|_{A^2_\rho}
$$
 for all $F \in A^2(\Omega,\rho)$ with $F(\zeta)=1$. In fact, the lower bound
 is sharp and uniquely attained by the function $F(z)=\frac{K(z,\zeta)}{K(\zeta,\zeta)}
$.

The theory of reproducing kernel Hilbert spaces guarantees that the existence of the reproducing kernel $K(\cdot,\cdot)$ is unique.

Recall that a holomorphic mapping $w=\Phi (z)$ from a domain $ \Omega_1$ to a domain $ \Omega_2$ is said to be biholomorphic if it is one-to-one, onto, and  its holomorphic
inverse exists.

 In fact, the kernel function with weight is biholomorphic invariant in the sense of the following Lemma.

\begin{lemma} Suppose that  $w=\Phi (z) $  is a biholomorphic mapping  of a domain $ \Omega_1$ onto a domain $ \Omega_2 $,   $\rho_1(z)$ and $\rho_2(w)$ are  two  positive continuous functions  on domains $ \Omega_1 $ and $\Omega_2 $  respectively,   $\rho_1 (z)=0$ for $z\not\in \Omega_1 $ and  $\rho_2(w)=0$ for $ w\not\in \Omega_2 $, $K_{\rho_1}(z,\zeta)$ and $K_{\rho_2}(w,\varsigma)$ are  reproducing kernels of two weighted Bergman spaces $A^2_{\rho_1}$ and $A^2_{\rho_2}$ respectively. If $\rho_1(z)=\rho_2(\Phi(z))$
for all $ z\in \Omega_1$, then
 \begin{eqnarray}
 K_{\rho_1}(z,\zeta)
 =(D\Phi)(z)K_{\rho_2}(\Phi(z),\Phi(\zeta))\overline{(D\Phi)(\zeta)},
 \label{riesz2}
\end{eqnarray}
where $(D\Phi)(z) $ is the determinant of the holomorphic Jacobian matrix of $w=\Phi (z) $.
\end{lemma}
\begin{proof}
Let $F\in A^2_{\rho_1}$, after a change of variables $\varsigma=\Phi (\zeta)$
in the following integral,
\begin{eqnarray*}
& &\int_{\Omega_1}(D\Phi)(z)K_{\rho_2}(\Phi(z),\Phi(\zeta))\overline{(D\Phi)(\zeta)}
F(\zeta)\rho_1(\zeta)dA(\zeta)\\
&=&
\int_{\Omega_2}(D\Phi)(z)K_{\rho_2}(\Phi(z), \varsigma)\overline{(D\Phi)(\Phi^{-1}(\varsigma))}
F(\Phi^{-1}(\varsigma))\rho_2(\varsigma)(D_R\Phi^{-1})(\varsigma)dA(\varsigma)
,
\end{eqnarray*}
where $D_R\Phi^{-1}$ is the determinant of the real Jacobian matrix of $\Phi^{-1}$. Based on the relationship between the determinant of the real Jacobian matrix and that of the holomorphic Jacobian matrix, i.e., $D_R\Phi^{-1}=|D\Phi^{-1}|^2$ (see \cite{SGK2} Proposition 1.4.10), the above formula simplifies to
\begin{equation*}
\int_{\Omega_2}(D\Phi)(z)K_{\rho_2}(\Phi(z), \varsigma)
\left \{ \left ((D\Phi)(\Phi^{-1}(\varsigma))\right )^{-1}
F(\Phi^{-1}(\varsigma))\right \}\rho_2(\varsigma)dA(\varsigma)
.
\end{equation*}
On the other hand, by hypothesis, the expression in braces is an element of $A^2_{\rho_2}$.
So the last line equals
\begin{eqnarray*}
(D\Phi)(z) \left ((D\Phi)(\Phi^{-1}(\Phi(z)))\right )^{-1}
F(\Phi^{-1}(\Phi(z)))=F(z).
\end{eqnarray*}
By the uniqueness of the reproducing kernel of the weighted Bergman space $A^2_{\rho_1}$, we see that (\ref{riesz2}) holds. This completes the proof of Lemma 3.
\end{proof}

\section{Main results}

In order to obtain the explicit reproducing kernel of the weighted Bergman kernel corresponding to a specific weights $\rho $ in a concrete domain, we suppose that $\rho (z) $ is  a positive continuous function  on a tube domain
$$
\Omega=T_B=\{z=x+iy: y\in B\}
$$
over an open connected subset $B$ of the real  $n$-dimensional Euclidean space $\mathbb R^n$. In addition, we assume that $ \rho (x+iy)=\rho (iy)$ for all $ x\in \mathbb{R}^n, y\in B$ and
 $\rho (z)=0$ for $z\not\in T_B $.
In this case,
the computation of the weighted  Bergman kernels on those tube domain is greatly benefited from the homogeneity of $T_B$ in the real directions.
An important tool is
the Laplace transforms $F=\mathscr L(f)$ of a function $f$,  that is,
\begin{equation}
F=(\mathscr Lf)(z)=\int_{\mathbb R^n}f(t)e^{2\pi iz\cdot t}dt, \label{laplacetransform}
\end{equation}
where $z\cdot t =x\cdot t+iy\cdot t=\sum\limits_{k=1}^nz_k\cdot t_k $, and $x\cdot t, y\cdot t$ are the Euclidean scalar products for $x,y, t\in \mathbb{R}^n$. The definition will be further justified together with the specific spaces that the test $f$ belongs to. It is obvious that $F$ is well defined only when $f$ decays sufficiently fast at $\infty$.

 Let
\begin{equation}
 I(t)=\int_B\rho(iy)e^{-4\pi y\cdot t}dy,\label{itdefinition}
\end{equation}
then the set $ U_I=\{t: I(t)<\infty \}$ is a convex set, and $ \log I(t) $  is a convex function on $U_I$. The weighted $L^p_I $ space is the set of the measurable function defined on $\mathbb R^n$ such that
 $$
 \|f\|_{L^p_I}=\left(\int_{\mathbb R^n}|f(t)|^pI(t)dt\right)^{\frac1p}<\infty.
 $$
 Notice that if $f\in L^p_I$, then $ f(t)=0$ almost everywhere for all $ t\not\in U_I$, so we can assume that the support of $f$ is contained in the closure of $U_I$.  We also see that $A^2_\rho$ and $L^2_I$ are Hilbert spaces with the inner product
  $$
  \langle F,G\rangle_{\rho}=\int_{B}\int_{\mathbb R^n}F(x+iy)\overline{G(x+iy)}\rho(iy)dxdy
  $$
   for $F,G\in A^2_\rho$ and
   $$
   \langle f,g\rangle_I=\int_{\mathbb R^n}f(t)\overline{g(t)}I(t)dt
   $$
    for $f,g\in L^2_I$ respectively.

The main result herein is established as follows.
\begin{theorem}
The weighted Bergman kernel $K(z,w)$ of $A^2_\rho$ is given by
\begin{equation}
K(z,w)=\int_{\mathbb R^n}e^{2\pi it(z-\overline w)}I^{-1}(t)dt,\label{kernelform}
\end{equation}
where $\rho (z) $ is  a positive continuous function  on the tube domain
$T_B$ and satisfies that  $ \rho (x+iy)=\rho (iy)$ for all $ x\in \mathbb{R}^n, y\in B$.
\end{theorem}

To prove Theorem 1, we need the following lemma, which is also an important result by itself.
\begin{lemma}
The Laplace transform $\mathscr L$ is an isometry from $L^2_I$ to $A^2_\rho$ preserving the Hilbert space norms, i.e.,
$$
\|\mathscr Lf\|_{A_{\rho}^2}=\|f\|_{L_I^2}.
$$
\end{lemma}
\begin{proof} First, we prove that if  $F(z)\in A^2_\rho$, there exists $f\in L^2_I$ such that $F(z)=(\mathscr Lf)(z)$, which means that the Laplace transform $\mathscr L$ is surjective.

       Let $B_{0}\subseteq B$ be a bounded connected open set, so there exists a positive constant $R_0$ such that $B_0\subseteq  D(0,R_0)$. Assume that $l_{\varepsilon}(z)=(1+\varepsilon(z_1^2+\cdots+z_n^2))^N$, where $N$ is a integer and $N>\frac n4$. Then for $\varepsilon\leq\frac{1}{2R_0^2}$ and $z=x+iy $ with $|y|\leq R_0$,
\begin{eqnarray*}
|l_{\varepsilon}(z)|&=&|((1+\varepsilon(z_1^2+\cdots+z_n^2))^2)^{\frac N2}| \\
&=&\left(\left(1+\varepsilon(|x|^2-|y|^2)\right)^2+4\varepsilon^2
\left(x\cdot y\right)^2\right)^{\frac N2} \\
&\geq & \left(1+\varepsilon(|x|^2-|y|^2)\right)^N\geq\left(\frac12+\varepsilon|x|^2\right)^N,
\end{eqnarray*}
i.e., $\left|l^{-1}_{\varepsilon}(z)\right|\leq\frac{1}{\left(\frac12+\varepsilon|x|^2\right)^N}$.
Set $F_{\varepsilon}(z)=F(z)l_{\varepsilon}^{-1}(z)$, then base on H\"{o}lder's inequality,
\begin{equation*}
\int_{\mathbb R^n}\left|F_{\varepsilon,y}(x)\right|dx\leq\left(\int_{\mathbb R^n}\left|F_{y}(x)\right|^2dx\right)^{\frac12}\left(\int_{\mathbb R^n}\left|l_{\varepsilon}^{-1}(x+iy)\right|^2dx\right)^{\frac12}<\infty,
\end{equation*}
which implies that $F_{\varepsilon,y}(x)=F_{\varepsilon}(x+iy)\in L^1(\mathbb R^n) $ and
\begin{equation*}
\log \int_{\mathbb R^n}\left|F_{\varepsilon}(x+iy)\right|dx
\end{equation*}
is a convex function on $B_0$. Therefore, for any compact $K\subseteq B_0$, we have
\begin{equation}
\sup \left\{\int_{\mathbb R^n}\left|F_{\varepsilon}(x+iy)\right|dx:y\in K \right\} <\infty.\label{laplacetransform1}
\end{equation}

For any $a,b,y\in B_0, t \in \mathbb{R}^n$, let
 \begin{eqnarray*}
 G_\varepsilon(z)&=&F_{\varepsilon}(z)e^{2\pi iz\cdot t};\\
 J(x',x_n,\tau )&=&F_\varepsilon(x+i(a+\tau(b-a))),\ \ \ x=(x',x_n), 0\leq \tau \leq 1; \\
 I_{\varepsilon}(y,t)&=& \check F_{\varepsilon,y}(t)e^{-2\pi y\cdot t}  .
 \end{eqnarray*}
In order to show that, for all $a,b \in B_0$,  $I_\varepsilon(a,t)=I_\varepsilon(b,t)$ and $I(a,t)=I(b,t)$ almost everywhere for all  $t\in \mathbb{R}^n$, we first assume that $a=(a',a_n)$, $b=(a',b_n)$ and the closed interval
 $[a,b]=\{a+t(b-a):0\leq t\leq 1\} $ is contained in $B_0$. Then (\ref{laplacetransform1})
  implies that the integral
\begin{equation*}
\int_0^{\infty}\int_0^1\int_{\mathbb R^{n-1}}(|J(x',x_n,\tau )| +|J(x',-x_n,\tau )|)dx'd\tau dx_n
\end{equation*}
is finite. This means
\begin{equation*}
\varliminf_{R\to\infty}\int_0^1\int_{\mathbb R^{n-1}}(|J(x',R,\tau )| +|J(x',-R,\tau )|)dx'd\tau=0.
\end{equation*}
Therefore,
\begin{eqnarray*}
 |I_\varepsilon(a,t)-I_\varepsilon(b,t)|&=&\left|\int_{\mathbb R^n}(G_\varepsilon(x+ib)-G_\varepsilon(x+ia))dx\right|\\
&=&\left|\int_{\mathbb R^n}\int_0^1\frac{\partial}{\partial \tau}G_\varepsilon(x+i(a+\tau (b-a) )d\tau dx\right|\\
&=&\left|\int_{\mathbb R^n}\int_0^1\frac{\partial}{\partial y_n}\left(G_\varepsilon(x+i(a',y_n))\right)\big|_{y_n=a_n+\tau(b_n-a_n)}(b_n-a_n)d\tau dx\right|\\
&=&|b_n-a_n|\left|\int_{\mathbb R^n}\int_0^1 i\frac{\partial}{\partial x_n}\left(G_{\varepsilon}(x+i(a+\tau(b-a)))\right)d\tau dx\right|\\
&\leq& C_1(t)\varliminf_{R\to\infty} \int_0^1\int_{\mathbb R^{n-1}}(|J(x',R,\tau )| +|J(x',-R,\tau )|)dx'd\tau =0,
\end{eqnarray*}
where $C_1(t)=|b_n-a_n|e^{-2\pi|t|(|a|+|b-a|)}$. Remark that $B_0$ is  connected and by an
iteration argument, we can show that $g_{\varepsilon,y}(t)=\check F_{\varepsilon,y}(t)e^{-2\pi y\cdot t}$ is a function independent of $y\in B_0$. Hence $g_{\varepsilon}(t)=g_{\varepsilon,y}(t)$ is independent of $y\in B_0$ and $g_{\varepsilon}(t)e^{2\pi y\cdot t}=\check F_{\varepsilon,y}(t)\in L^1(\mathbb R^n)$
for all $y\in B_0$.

On the other hand, it is obvious that $F_{\varepsilon,y}(x)\to F_{y}(x)$ as $\varepsilon\to0$. Let $Iy(t)= \check F_{y}(t)e^{-2\pi y\cdot t} $, we can also prove that $\check F_{y}(t)e^{-2\pi y\cdot t} $ is independent of $y\in B_0$ and $ g(t)= \check F_{y}(t)e^{-2\pi y\cdot t} $ almost everywhere. Indeed, for $a,b\in B_0$ and any compact subset $K\subset \mathbb R^n$,
let $R_1=\max \{|t|:t\in K\}$. Then Plancherel's theorem implies that
\begin{eqnarray*}
\|I_b-I_a\|_{L^2(K)}
&\leq& \|I_b-g_\varepsilon\|_{L^2(K)}+\|g_\varepsilon-I_a\|_{L^2(K)} \\
&\leq&e^{2\pi R_0 R_1}\left (\|\check F_{a}-\check F_{\varepsilon,a}\|_{L^2(K)}+\|\check F_{\varepsilon,b}-\check F_{b}\|_{L^2(K)}\right )\\
&=&e^{2\pi R_0 R_1}\left (\|F_{a}- F_{\varepsilon,a}\|_{L^2(\mathbb{R}^n)}+\|F_{b}- F_{\varepsilon,b}\|_{L^2(\mathbb{R}^n)}\right )\to0
\end{eqnarray*}
as $\varepsilon\to0$. Hence  $\check F_{a}(t)e^{-2\pi a\cdot t}=\check F_{b}(t)e^{-2\pi b\cdot t}$ almost everywhere on $\mathbb R^n$.

Next, we show that $g(t)e^{2\pi y\cdot t}\in L^1(\mathbb R^n)$. In order to prove this affirmance, decompose $\mathbb R^n$ into the union of finite non-overlapping cones  $\{\Gamma_k\}_{k=1}^{\infty}$ with common vertex at the origin, i.e., $\mathbb R^n=\cup_{k=1}^{N}\Gamma_k$ and let $B_{\delta}=\overline{D(y_0,\delta)}\subset B_0$. Then for any $y\in D(y_0,\frac{\delta}{4})$ and $y_k\in (y_0+\Gamma_k)$ satisfying $\frac{3\delta}{4}\leq|y_k-y_0|<\delta$, we have $(y_k-y)\cdot t\geq\frac{|y_k-y_0|}{\sqrt 2}|t|-|y_0-y||t|\geq(\frac{3}{4\sqrt 2}-\frac14)\delta|t|\geq\frac14\delta|t|$ for $y_k-y_0,t\in\Gamma_k$. Thus, it follows from H\"{o}lder's inequality and Plancherel's theorem that
\begin{eqnarray*}
\int_{\Gamma_k}|g(t)e^{2\pi y\cdot t}|dt&\leq&\int_{\Gamma_k}|\check F_{y_k}(t)e^{-\pi\frac{\delta_0}{4}|t|}|dt\leq\left(\int_{\Gamma_k}|\check F_{y_k}(t)|^2dt\right)^{\frac12}\left(\int_{\Gamma_k}| e^{-2\pi\frac{\delta_0}{4}|t|}|dt\right)^{\frac12}\\
&=&\left(\int_{\Gamma_k}| F_{y_k}(t)|^2dt\right)^{\frac12}\left(\int_{\Gamma_k}| e^{-2\pi\frac{\delta_0}{4}|t|}|dt\right)^{\frac12}<\infty,
\end{eqnarray*}
which shows that $g(t)e^{2\pi y\cdot t}\in L^1(\Gamma_k)$. Then $g(t)e^{2\pi y\cdot t}\in L^1(\mathbb R^n)$. Therefore, we can see that the function $G(z)=\int_{\mathbb R^n}g(t)e^{-2\pi i(x+iy)\cdot t}dt$ is well-defined and holomorphic on the tube domain $T_{B_0} $.

Now we can prove that
$$
 \lim_{\varepsilon\to0}\int_{\mathbb R^n}g_{\varepsilon}(t)e^{-2\pi i(x+iy)\cdot t}dt= \int_{\mathbb R^n}g(t)e^{-2\pi i(x+iy)\cdot t}dt.
 $$
  For $y\in B_{0}$,
\begin{eqnarray*}
&&\left|\int_{\mathbb R^n}(g_{\varepsilon}(t)-g(t))e^{-2\pi i(x+iy)\cdot t}dt\right|\leq\int_{\mathbb R^n}\left|\left(\check F_{\varepsilon,y}(t)e^{-2\pi y\cdot t}-\check F_{y}(t)e^{-2\pi y\cdot t}\right)e^{2\pi iz\cdot t}\right|dt\\
&=&\sum_{k=1}^n\int_{\Gamma_k}\left|\left(\check F_{\varepsilon,y_k}(x)-\check F_{y_k}(x)\right)e^{-2\pi i(y_k-y)\cdot  t}\right|dt\\
&\leq&\sum_{k=1}^n\left(\int_{\Gamma_k}|\check F_{\varepsilon,y_k}(x)-\check F_{y_k}(x)|^2dt\right)^{\frac12} \left(\int_{\Gamma_k}|e^{-2\pi\frac{\delta_0}{4}|t|}|dt\right)^{\frac12}\\
&\leq & \left [\max\limits_{1\leq k\leq n} \left(\int_{\Gamma_k}|e^{-2\pi\frac{\delta_0}{4}|t|}|dt\right)^{\frac12}\right]\sum_{k=1}^n\left(\int_{\Gamma_k}| F_{\varepsilon,y_k}(x)-F_{y_k}(x)|^2dt\right)^{\frac12}\to0
\end{eqnarray*}
as $\varepsilon\to0$. It follows that $\lim\limits_{\varepsilon\to0}F_{\varepsilon}(z)=G(z)$. Combining with $\lim\limits_{\varepsilon\to0}F_{\varepsilon}(z)=F(z)$, we get $G(z)=F(z)$ for $z\in T_{B_0}$. Therefore, there exists a measurable function $g(t)$ such that $g(t)e^{2\pi y\cdot t}\in L^1(\mathbb R^n)$ for all $y\in B_0$, then
$$
F(z)=\int_{\mathbb R^n}g(t)e^{-2\pi iz\cdot t}dt
$$
holds for $z\in T_{B_0}$. Hence $g(t)=\check F_y(t)e^{-2\pi y\cdot t}$ for all $y\in B_0$.
Since $B$ is connected, we can choose a sequence of bounded domains $\{B_k\}$ such that $B_0\subset B_1\subset B_2\subset\ldots$ and  $B=\cup_{k=0}^{\infty}B_{\delta_k}$. Then
$\check F_{y_k}(t)e^{-2\pi y_k\cdot t}=\check F_y(t)e^{-2\pi y\cdot t}$ for $y\in B_0$ and $y_k\in B_k$, where $k>0$. These imply that  $g(t)=\check F(t)e^{-2\pi y\cdot t}$ holds for all $y\in B $.
In other words, $F(z)=\int_{\mathbb R^n}g(t)e^{-2\pi iz\cdot t}dt$ holds for all $z \in T_B$. By letting $f(t)=g(-t)$, we see that  $f\in L^2_I $ and  $F(z)=(\mathscr Lf)(t)$ for any given $F(z)\in A^2_\rho$, which means that the Laplace transform $\mathscr L$ is surjective.

Now we prove that (\ref{laplacetransform}) is well-defined, injective and preserves norm, i.e., $\|\mathscr Lf\|_{A^2_\rho}=\|f\|_{L^2_I}$. Let $F(x+iy)= F_y(x)=(\mathscr L f)(x+iy)$ for every fixed $y\in B$. Based on Plancherel's theorem, there holds
$$
\int_{\mathbb R^n}|F_y(x)|^2dx=\int_{\mathbb R^n}|f(t)|^2e^{-4\pi y\cdot t }dt.
$$
Multiplying by $\rho(iy )$ and performing integral over $B$ on both sides of the above equation, we have
\begin{eqnarray*}
\|F\|_{A^2_\rho}^2&=&\int_B\left(\int_{\mathbb R^n}|F(x+iy)|^2dx\right)\rho(iy)dy\\
&=&\int_B\left(\int_{\mathbb R^n}|e^{-2\pi y\cdot t}f(t)|^2dt\right)\rho(iy)dy\\
&=&\int_{\mathbb R^n}|f(t)|^2I(t)dt
=\|f\|^2_{L^2_I}.
\end{eqnarray*}

It then follows that $ f(t)=0$ almost everywhere for all $ t\not\in U_I$, where $U_I=\{t:I(t)<\infty\}$. Therefore, $f(t)$ is supported in $U_I$.
This completes the proof of Lemma 4.
\end{proof}

Let us now prove Theorem 1.
\begin{proof}
For $F(z)\in A^2_\rho( T_B)$, there exists $f\in L^2_I$ such that $F(z)=(\mathscr Lf)(z)=\int_{\mathbb R^n}f(t)e^{2\pi iz\cdot t}dt$. And for the kernel $K_{z_0}(z)=K(z,z_0)\in A^2_\rho$, there also exists $f_{z_0}\in L^2_I$ such that $K_{z_0}(z)=(\mathscr Lf_{z_0})(z)=\int_{\mathbb R^n}f_{z_0}(t)e^{2\pi iz\cdot t}dt$ for $z_0,z\in T_B$.
On the other hand, Lemma 1 claims that $\mathscr L$ is an isometry from $L^2_I$ to $A^2_\rho$ preserving the Hilbert space norm. Using the polarization identity of $\|F\|_{A^2_\rho}=\|f\|_{L^2_I}$, it then follows that the inner product is also preserved. Hence we have
\begin{equation*}
F(z_0)=\langle F,K_{z_0}\rangle_{\rho}=\langle f,f_{z_0}\rangle_{I}=\int_{\mathbb R^n}f(t)\overline{f_{z_0}(t)}I(t)dt.
\end{equation*}
Hence
\begin{equation*}
\int_{\mathbb R^n}f(t)e^{2\pi iz_0\cdot t}dt=\int_{\mathbb R^n}f(t)\overline{f_{z_0}(t)}I(t)dt
\end{equation*}
holds for every $f\in L^2_I$, which implies that $e^{2\pi iz_0\cdot t}=\overline{f_{z_0}(t)}I(t)$ almost everywhere on $U_I=\{t\in\mathbb R^n:I(t)<\infty\}$. Then, $f_{z_0}(t)=e^{-2\pi i\bar{z_0}\cdot t}I^{-1}(t)$. Here, $I(t)^{-1}$ takes $0$ when $I(t)=\infty$ by the definition of $I(t)$.
Hence,
\begin{equation*}
K_{z_0}(z)=K(z,z_0)=\int_{\mathbb R^n}f_{z_0}(t)e^{2\pi iz\cdot t}dt=\int_{\mathbb R^n}e^{2\pi i(z-\bar{z_0})\cdot t}I^{-1}(t)dt.
\end{equation*}
\end{proof}
Note that the Bergman kernel is uniquely characterized by the following three properties.
(i) $K(z,z_0)=\overline{K(z_0,z)}$ for all $z,z_0\in T_B$;\\
(ii) $K(z,z_0)$ reproduces every element in $A^2_\rho$ in the following sense \begin{equation*}
F(z)=\int_B\int_{\mathbb R^n}K(z,u+iv)F(u+iv)\rho(iv)dudv
\end{equation*}
for every $F\in A^2_\rho$;\\
(iii)$K_{z_0}\in A^2_\rho$ for all $z_0\in T_B$, where $K_{z_0}(z)=K(z,z_0)$.

We shall show that (\ref{kernelform}) admits these properties.
We first prove the symmetric property as follows,
\begin{eqnarray*}
\overline{K(z_0,z)}&=&\overline{\int_{\mathbb R^n}e^{2\pi i(z_0-\bar{z})\cdot t}I^{-1}(t)dt}=\int_{\mathbb R^n}e^{-2\pi i(\bar{z_0}-z)\cdot t}\overline{I^{-1}(t)}dt\\
&=&\int_{\mathbb R^n}e^{2\pi i(z-\bar{z_0})\cdot t}I^{-1}(t)dt=K(z,z_0),
\end{eqnarray*}
which means (i) holds for the Bergman kernel in the form of (\ref{kernelform}). We then show $K(z,z_0)$ reproduces every element in $A^2_\rho$. Indeed, for $F(z),K_{z_0}(z)\in A^2_\rho$,
\begin{equation*}
F(z)=\int_{U_I}f(t)e^{2\pi iz\cdot t}dt
\end{equation*}
and
\begin{equation*}
K_{z_0}(z)=\int_{\mathbb R^n}\left(\frac{e^{-2\pi i\bar{z_0}\cdot t}}{I(t)}\right)e^{2\pi it\cdot z}dt.
\end{equation*}
Then the polarization identity implies that
\begin{eqnarray*}
\langle F,K_{z_0}\rangle_{A^2_{\rho}}&=&\int_{B}\int_{\mathbb R^n}K(z_0,z)F(z)dA_{\rho}(z)\\
&=&\left \langle f(t),\frac{e^{-2\pi i\bar{z_0}\cdot t}}{I(t)}\right \rangle_{L^2_I}\\
&=&\int_{U_I}f(t)\overline{\frac{e^{-2\pi i\bar{z_0}\cdot t}}{I(t)}}I(t)dt\\
&=&\int_{U_I}f(t)e^{2\pi iz_0\cdot t}dt\\
&=&F(z_0).
\end{eqnarray*}
Hence, the second property is proved.

Finally, we prove that $K_{z_0}(z)\in A^2_\rho $. For fixed $z_0=u+iv\in T_B$, there exists $\delta>0$ such that $v+P_{\delta}\subset B$, where $P_{\delta}=[-\delta,\delta]^n\subset\mathbb R^n$. Let $\varepsilon=\min\{\rho(iy):y\in v+P_{\delta}\}>0$, then
\begin{equation*}
I(t)e^{4\pi v\cdot t}=\int_B e^{-4\pi(y-v)\cdot t}\rho(y)dy\geq \varepsilon\int_{P_{\delta}}e^{-4\pi \eta\cdot t}d\eta=\varepsilon\prod_{k=1}^n\frac{\sinh(4\pi\delta t_k)}{2\pi\delta t_k}.
\end{equation*}
Therefore, again by the polarization identity,
\begin{eqnarray*}
\langle K_{z_0},K_{z_0}\rangle_{A^2_{\rho}}&=&\int_{T_B}|K_{z_0}(z)|^2dA_{\rho}(z)=\langle K_{z_0},K_{z_0}\rangle_{A^2_{\rho}}\\
&=&\left \langle \frac{e^{-2\pi i\bar{z_0}\cdot t}}{I(t)},\frac{e^{-2\pi i\bar{z_0}\cdot t}}{I(t)}\right \rangle_{L^2_I}=\int_{\mathbb R^n}\left|\frac{e^{-2\pi i\bar{z_0}\cdot t}}{I(t)}\right|^2I(t)dt\\
&=&\int_{\mathbb R^n}\frac{e^{-4\pi v\cdot t}}{I(t)}dt\leq\int_{\mathbb R^n}\frac{1}{\varepsilon}\left(\prod_{k=1}^n\frac{\sinh(4\pi\delta t_k)}{2\pi\delta t_k}\right)^{-1}dt<\infty.
\end{eqnarray*}
Therefore, $K_{z_0}(z)\in A^2_\rho$ for $z \in T_B$.

\section{Computation of some weighted Bergman kernels}

The weighted Bergman kernel have never been computed explicitly. However, the theory
would be a bit hollow if we do not compute at least one weighted Bergman kernel. In this section, we calculate some weighted Bergman kernels as examples.
The first example is the weighted Bergman kernel for tube over the following cone.

{\bf Example 1.} Suppose that $B=\{y=(y',y_n):y_n>|y'|^2\}$. Denote by $A^2_\alpha$ the space of analytic  functions on tube domain $\Omega=\{z=x+iy:x\in \mathbb{R}^n, y\in B\} $ such that
\begin{equation}
\|F\|_{A^2_{\alpha}}=\left(\int_{\Omega}|F(x+iy)|^2(y_n-|y'|^2)^{\alpha}dxdy\right)^{\frac1{2}}<\infty.\label{aanorm}
\end{equation}
 Then  the reproducing kernel for the Hilbert space $A^2_\alpha $ is
\begin{equation}
K_\alpha(z,w)=C_{1,\alpha} ((z'-\bar{w'})^2-2i (z_n-\bar{w}_{n}))^{-n-\alpha-1},\label{laplacetransform2}
\end{equation}
where $z'=(z_1,\ldots,z_{n-1})$, $w'=(w_1,\ldots,w_{n-1})$ and
\begin{equation}
C_{1,\alpha}=\frac{2^{n+1+2\alpha}\Gamma(n+\alpha+1)}{\Gamma(\alpha+1)\pi^{n}}.\label{laplacetransform2-1}
\end{equation}
Here, $\Re \{(z'-\bar{w'})^2-2i (z_n-\bar{w}_{n})\}=2(y_n+v_{n})+(x'-u')\cdot(x'-u')-(y'+v')\cdot (y'+v')>0$, then $|\arg ((z'-\bar{w'})^2-2i (z_n-\bar{w}_{n}))|<\frac{\pi}{2}
$, in which $ z=(z',z_n)=(x',x_n)+i(y',y_n)$, $ w=(w',w_n )=(u',u_n)+i(v',v_n)\in T_B$ and $z'^2=z'\cdot z'=z^2_1+z^2_2+\ldots+z^2_{n-1}$.

\begin{proof}
We fist compute $I(t)$ as follows,
\begin{equation}
I(t)=\int_Be^{-4\pi y\cdot t}(y_n-|y'|^2)^{\alpha}dy=\int_0^{\infty}\int_{|y'|^2<y_n}(y_n-|y'|^2)^{\alpha}e^{-4\pi(y_nt_n+y'\cdot t')}dy'dy_n.\label{icompute}
\end{equation}
 Performing variable substitution on (\ref{icompute}) with  $x_n=y_n-|y'|^2$, $x'=y'$, then  $y_n=x_n+|x'|^2$. We then obtain
\begin{eqnarray*}
I(t)&=&\int_0^{\infty} \int_{\mathbb R^{n-1}}x_n^{\alpha} e^{-4\pi(x_n+|x'|^2)t_n}e^{-4\pi x'\cdot t'}dx'dx_n\\
&=&\int_0^{\infty}x_n^{\alpha}e^{-4\pi x_nt_n}\left(\prod_{k=1}^{n-1}\int_{\mathbb R}e^{-4\pi x_k^2t_n-4\pi x_kt_k}dx_k\right)dx_n.
\end{eqnarray*}
It is obvious that $I(t)=\infty$ when $t_n\leq0$, For $t_n>0 $,
\begin{eqnarray*}
I(t)&=&\int_0^{\infty}x_n^{\alpha}e^{-4\pi x_nt_n}\left(\prod_{k=1}^{n-1}\int_{\mathbb R}e^{-4\pi x_k^2t_n-4\pi x_kt_k}dx_k\right)dx_n\\
&=&\int_0^{\infty}x_n^{\alpha}e^{-4\pi x_nt_n}\left(\prod_{k=1}^{n-1}\int_{\mathbb R}e^{-4\pi t_n\left(x_k^2-\frac{x_kt_k}{t_n}+\left(\frac{t_k}{2t_n}\right)^2- \left(\frac{t_k}{2t_n}\right)^2\right)}dx_k\right)dx_n\\
&=&\int_0^{\infty}x_n^{\alpha}e^{-4\pi x_nt_n}\left(\prod_{k=1}^{n-1}\int_{\mathbb R}e^{-4\pi t_n\left(x_k-\frac{t_k}{2t_n}\right)^2}e^{ \frac{ \pi t_k^2}{t_n}}dx_k\right)dx_n.
\end{eqnarray*}
Let $s=x_k-\frac{t_k}{2t_n}$ and $4\pi t_ns^2=\eta^2$, then
\begin{eqnarray*}
I(t)&=&\int_0^{\infty}x_n^{\alpha}e^{-4\pi x_nt_n}\left(\prod_{k=1}^{n-1}e^{ \frac{ \pi t_k^2}{t_n}}\int_{\mathbb R}e^{-4\pi t_ns^2}ds\right)dx_n\\
&=&\int_0^{\infty}x_n^{\alpha}e^{-4\pi x_nt_n}\left(\prod_{k=1}^{n-1}\frac{ e^{ \frac{ \pi t_k^2}{t_n}}}{2\sqrt{\pi t_n}}\int_{\mathbb R}e^{-\eta^2}d\eta\right)dx_n\\
&=&\int_0^{\infty}x_n^{\alpha}e^{-4\pi x_nt_n}e^{ \frac{ \pi |t'|^2}{t_n}}\left(\frac{\sqrt{\pi}}{2\sqrt{\pi t_n}}\right)^{n-1}dx_n\\
&=&e^{ \frac{ \pi |t'|^2}{t_n}}(2\sqrt{t_n})^{1-n}\int_0^{\infty}x_n^{\alpha}e^{-4\pi x_nt_n}dx_n.
\end{eqnarray*}
Putting $u=-4\pi x_nt_n$, we have
\begin{eqnarray*}
I(t)&=&e^{ \frac{ \pi |t'|^2}{t_n}}(2\sqrt{t_n})^{1-n}\int_0^{\infty}\frac{1}{(4\pi t_n)^{\alpha+1}}u^{\alpha}e^{u}du\\
&=&\frac{2^{1-n}\Gamma(\alpha+1)}{(4\pi )^{\alpha+1}}e^{ \frac{ \pi |t'|^2}{t_n}}t_n^{\frac12(1-n)-\alpha-1}.
\end{eqnarray*}

Now we can compute the formula of weighted Bergman kernel. According to the representation form of reproducing kernel,
\begin{eqnarray*}
K_{\alpha}(z,z_0)=\int_{\mathbb R^n}e^{2\pi i(z-\bar{z_0})\cdot t}I^{-1}(t)dt=\frac{(4\pi )^{\alpha+1}2^{n-1}}{\Gamma(\alpha+1)}\int_{\mathbb R^n}e^{2\pi i(z-\bar{z_0})\cdot t-\frac{ \pi |t'|^2}{t_n}}t_n^{\frac12(n-1)+\alpha+1}dt.
\end{eqnarray*}
Set $w=z-\bar{z_0}$ and $C_{\alpha}=\frac{(4\pi )^{\alpha+1}2^{n-1}}{\Gamma(\alpha+1)}$, then
\begin{eqnarray*}
K_{\alpha}(z,z_0)&=&C_{\alpha}\int_{\mathbb R^n}e^{2\pi iw\cdot t-\frac{ \pi |t'|^2}{t_n}}t_n^{\frac12(n-1)+\alpha+1}dt\\
&=&C_{\alpha}\int_0^{\infty}\int_{\mathbb R^{n-1}}e^{2\pi i(w_nt_n+w'\cdot t')-\frac{ \pi |t'|^2}{t_n}}t_n^{\frac12(n-1)+\alpha+1}dt'dt_n\\
&=&C_{\alpha}\int_0^{\infty}e^{2\pi i w_nt_n }t_n^{\frac12(n-1)+\alpha+1}\left(\int_{\mathbb R^{n-1}}e^{2\pi i w'\cdot t' -\frac{ \pi |t'|^2}{t_n}}dt'\right)dt_n\\
&=&C_{\alpha}\int_0^{\infty}e^{2\pi i w_nt_n }t_n^{\frac12(n-1)+\alpha+1}\left(\prod_{k=1}^{n-1}\int_{\mathbb R}e^{2\pi i w_kt_k -\frac{ \pi t_k^2}{t_n}}dt_k\right)dt_n,
\end{eqnarray*}
in which
\begin{eqnarray*}
\int_{\mathbb R}e^{2\pi i w_kt_k -\frac{ \pi t_k^2}{t_n}}dt_k &=&\int_{\mathbb R}e^{ -\frac{ \pi}{t_n}\left(t_k^2-2i w_kt_nt_k+\left(iw_kt_n\right)^2+w_k^2t_n^2\right)}dt_k\\
 &=&e^{-\frac{ \pi w_k^2t_n^2}{t_n}}\int_{\mathbb R}e^{ -\frac{ \pi}{t_n}\left(t_k-iw_kt_n\right)^2}dt_k.
\end{eqnarray*}
Let $s= t_k-iw_kt_n$ and $\eta^2=\frac{ \pi}{t_n}s^2$, then
\begin{eqnarray*}
\int_{\mathbb R}e^{2\pi i w_kt_k -\frac{ \pi t_k^2}{t_n}}dt_k
=e^{-\frac{ \pi w_k^2t_n^2}{t_n}}\int_{\mathbb R}e^{ -\frac{ \pi}{t_n}s^2}ds
=e^{-\frac{ \pi w_k^2t_n^2}{t_n}}\int_{\mathbb R}\left(\frac{t_n}{\pi}\right)^{\frac12}e^{ -\eta^2}d\eta
=e^{-\frac{ \pi w_k^2t_n^2}{t_n}}t_n^{\frac12}.
\end{eqnarray*}
Therefore,
\begin{eqnarray*}
K_{\alpha}(z,z_0)&=&C_{\alpha}\int_0^{\infty}e^{2\pi i w_nt_n }t_n^{\frac12(n-1)+\alpha+1}\prod_{k=1}^{n-1}\left(e^{-\frac{ \pi w_k^2t_n^2}{t_n}}t_n^{\frac12}\right)dt_n\\
&=&C_{\alpha}\int_0^{\infty}e^{2\pi i w_nt_n }t_n^{\frac12(n-1)+\alpha+1}e^{-\frac{ \pi \sum_{k=1}^{n-1}w_k^2t_n^2}{t_n}}t_n^{\frac12(n-1)}dt_n\\
&=&C_{\alpha}\int_0^{\infty}e^{\pi t_n(2i w_n- w'\cdot w')}t_n^{n+\alpha}dt_n.
\end{eqnarray*}
In order to compute $K_{\alpha}(z,z_0)$, it suffices to show that $\Re(2i w_n- w'\cdot w')<0$, where $w=z-\bar{z_0}$ for $z,z_0 \in T_B$. In fact, since $\Re(iw_n)=-\Im w_n$, $\Re (w'\cdot w')=(\Re w' )^2-(\Im w' )^2$, for $w_n=z_n-\bar{z}_{0,n}$,
\begin{eqnarray*}
\Re(2i w_n- w'\cdot w')&=&-2\Im w_n-\sum_{k=1}^{n-1}\left((\Re w_k)^2-(\Im w_k)^2\right)\\
&=&-2(y_n+y_{0,n})+\sum_{k=1}^{n-1}(y_k+y_{0,k})^2-\sum_{k=1}^{n-1}(x_k-x_{0,k})^2\\
&\leq& -2(y_n+y_{0,n})+\sum_{k=1}^{n-1}(y_k^2+y_{0,k}^2+2 y_k y_{0,k})\\
&\leq& -2(y_n+y_{0,n})+2\sum_{k=1}^{n-1}(y_k^2+y_{0,k}^2)\\
&=&2\left(\left(-y_n+\sum_{k=1}^{n-1}y_k^2\right) +\left(-y_{0,n}+\sum_{k=1}^{n-1}y_{0,k}^2\right)\right)<0
.\end{eqnarray*}
Now, we can continue to calculate $K_{\alpha}(z,z_0)$. Let $u=\pi t_n(w'\cdot w'-2i w_n )$,
\begin{eqnarray*}
K_{\alpha}(z,z_0)&=&C_{\alpha}\int_0^{\infty}e^{-\pi t_n(w'\cdot w'-2i w_n )}t_n^{n+\alpha}dt_n\\
&=&C_{\alpha}\int_0^{\infty}\frac{1}{(\pi(w'\cdot w'-2i w_n ))^{n+\alpha+1}}e^{-u}u^{n+\alpha}du\\
&=&\frac{C_{\alpha}\Gamma(n+\alpha+1)}{(\pi(w'\cdot w'-2i w_n ))^{n+\alpha+1}}\\
&=& C_{\alpha,1}(w'\cdot w'-2i w_n )^{-n-\alpha-1}\\
&=&C_{\alpha,1}(z'-\bar{z}_0')^2-2i (z_n-\bar{z}_{0,n}))^{-n-\alpha-1},
\end{eqnarray*}\\
which shows that (\ref{laplacetransform2}) holds.
\end{proof}

 {\bf Example 2.} Denote by $\Omega_n $ the Siegel domain in $\mathbb C^n$, and define it as
$$
\Omega_n =\{z=x+iy:y_n>|z'|^2\},
$$
where $z=(z',z_n), z_n=x_n+iy_n.$ Let $\alpha \in \mathbb{R}$ and denote by  $A_{\alpha}^2(\Omega_n)$ the space of analytic  functions $F(z)$ in the Siegel  domain $\Omega_n $ satisfying
\begin{equation}
\|F\|_{A^2_{\alpha}(\Omega_n)}=\left(\int_{\Omega_n}|F(x+iy)|^2(y_n-|z'|^2)^{\alpha}dxdy\right)^{\frac1{2}}<\infty.\label{bnorm}
\end{equation}
 Then  the reproducing kernel for the Hilbert space $A^2_\alpha (\Omega_n)$ is
\begin{equation}
K_{\alpha,\Omega_n}(z,w)=C_{2,\alpha}( i (\bar{w}_n-z_{n})-2z'\cdot \bar{w}')^{-n-\alpha-1}\label{laplacetransform3},
\end{equation}
where $C_{2,\alpha}=2^{-2-\alpha} C_{1,\alpha}$, $ C_{1,\alpha}$  is defined by (\ref{laplacetransform2-1}).

\begin{proof}
Define a transform $\Phi:\Omega_n\to T_B$ as $w=(w',w_n)=\Phi(z)=(2^{\frac12}z',z_n-iz'\cdot z')$. Then we observe that the Siegel domain $\Omega_n $  is biholomorphically equivalent to the tube domain $T_B$ over $B=\{v\in\mathbb R^n:v_n>|v'|^2\}$ via $\Phi$. The inverse of $\Phi$ is $z=\Phi^{-1}(w)=(2^{-\frac12}w',w_n+\frac{i}{2}w'\cdot w')$, and the determinants of the holomorphic Jacobian matrixes of $w=\Phi (z) $ and $z=\Phi^{-1}(w) $ are $(D\Phi)(z)=2^{\frac12(n-1)}$ and $(D\Phi^{-1})(w)=2^{-\frac12(n-1)}$ respectively.
For $w=u+iv$, $ z=x+iy \in T_B$,
\begin{eqnarray*}
y_n-|z'|^2&=&y_n-\sum_{k=1}^{n-1}(x_k^2+y_k^2)\\
&=&\Im(w_n+\frac{i}{2}w'\cdot w')-\frac12|w'|^2\\
&=&v_n+\frac12(u'\cdot u'-v'\cdot v')-\frac12|w'|^2\\
&=&v_n+\frac12\sum_{k=1}^{n-1}(u_k^2-v_k^2)-\frac12\sum_{k=1}^{n-1}(u_k^2+v_k^2)\\
&=&v_n-|v'|^2.
\end{eqnarray*}
By Lemma 3, the reproducing kernel $K_{\alpha, \Omega_n} $ of the  Hibert space $A^2_\alpha (\Omega_n)$ is
\begin{equation}
K_{\alpha, \Omega_n}(z,w)=2^{n-1}K_{\alpha}(\Phi(z),\Phi(w))=C_{1,\alpha}2^{-2-\alpha}( i (\bar{w}_n-z_{n})-2z'\cdot \bar{w}')^{-n-\alpha-1},\label{laplacetransform3}
\end{equation}
where $ C_{1,\alpha} $ is defined by (\ref{laplacetransform2-1}).
\end{proof}

{\bf Example 3.}  For $\alpha \in \mathbb{R}$, denote by $B_n$ the unit ball  $\{z:|z|<1\}$ and $A_{\alpha}^2(B_n)$ the space of analytic  functions $F(z)$ on $B_n $ satisfying
\begin{equation}
\|F\|_{A^2_{\alpha}(B_n)}=\left(\int_{B_n}|F(z)|^2\frac{(1-|z|^2)^{\alpha}}{|1+z_n|^{2\alpha} }dxdy\right)^{\frac1{2}}<\infty.\label{anorm-2}
\end{equation}
 Then  the reproducing kernel for the Hilbert space $A^2_\alpha (B_n)$ is
\begin{equation}
K_{\alpha,B_n}(z,w)=C_{3,\alpha}(1+\bar{w}_n)^\alpha(1+z_n)^\alpha (1-\bar w\cdot z)^{-n-\alpha-1},\label{laplacetransform-2}
\end{equation}
with $C_{3,\alpha}=2^{1-3\alpha-n} C_{2,\alpha}=2^{-4\alpha-n-1} C_{1,\alpha}$, where $C_{1,\alpha}$  is defined by (\ref{laplacetransform2-1}).

\begin{proof}
Note that the Siegel domain $\Omega_n $  is an unbounded realization of $B_n$, i.e., $\Omega_n$ is  biholomorphically  equivalent to $B_n$. The corresponding biholomorphic  automorphism (so-called the Cayley transform) $ w=\Phi (z): B_n  \mapsto \Omega_n$ can be written in the following explicit form,
 \begin{equation*}
w=\Phi(z)=\left(\frac{2z'}{z_n+1},\frac{4i(1-z_n)}{1+z_n}\right)\ \mbox{for}\ z\in B_n.\ \
\end{equation*}
And its inverse form is
 \begin{equation*}
z=\Phi^{-1}(w)=\left(\frac{w'}{1-\frac{i}{4}w_n},\frac{1+\frac{i}{4}w_n}{1-\frac{i}{4}w_n}\right) \ \mbox{for}\ w\in \Omega_n.
\end{equation*}
 The determinants of the holomorphic Jacobian matrixes of $w=\Phi (z) $ and $z=\Phi^{-1}(w) $ are
\begin{equation*}
(D\Phi)(z)=\frac{-i2^{n+2}}{(1+z_n)^{n+1}} \mbox{ and } (D\Phi^{-1})(w)=\frac{i}{2(1-\frac{i}{4})^{n+1}}
\end{equation*}
respectively.
By Lemma 3, the reproducing kernel $K_{\alpha, B_n} $ of the  Hibert space $A^2_\alpha (B_n)$ is
\begin{eqnarray}
K_{\alpha, B_n}(z,w)&=&\frac{-i2^{n+2}}{(1+z_n)^{n+1}}K_{\alpha, \Omega_n}(\Phi(z),\Phi(w))\frac{i2^{n+2}}{(1+\bar{w}_n)^{n+1}}\nonumber\\
&=&C_{3,\alpha}(1+\bar{w}_n)^\alpha(1+z_n)^\alpha (1-\bar w\cdot z)^{-n-\alpha-1}\label{laplacetransform3-2}
\end{eqnarray}
with $C_{3,\alpha}=2^{1-3\alpha-n} C_{2,\alpha}=2^{-4\alpha-n-1} C_{1,\alpha}$, where $C_{1,\alpha}$  is defined by (\ref{laplacetransform2-1}).
\end{proof}

{\bf Example 4.}
Suppose that $\Lambda_n$ is the Lorentz cone (or  a forward light cone) defined by $\{y=(y',y_n):y_n>|y'|\}$. The quadratic function  $\Delta(y)=y_n^2-|y'|^2 $ is call the Lorentz form. Let $\alpha\in\mathbb R$ and denote by  $A_{\alpha}^2(T_{\Lambda_n})$ the space of analytic  functions $F(z)$ in the tube domain $T_{\Lambda_n}$ over the forward cone $\Lambda_n $ such that
\begin{equation}
\|F\|_{A^2_{\alpha}(\Lambda_n)}=\left(\int_{\Lambda_n}\int_{\mathbb{R}^n}|F(x+iy)|^2(\Delta(y))^{\alpha}dxdy\right)^{\frac1{2}}<\infty.\label{anorm-3}
\end{equation}
 Then  the reproducing kernel for the Hilbert space $A^2_\alpha (T_{\Lambda_n})$ is
\begin{equation}
K_{\alpha,\Lambda_n}(z,w)=C_{4,\alpha}P(z-\bar w)^{-\alpha-n},\label{laplacetransform4}
\end{equation}
where $C_{4,\alpha}$  is defined by
\begin{equation}
C_{4,\alpha}=\frac{4^{\alpha}\Gamma(\alpha+\frac n2+1)\Gamma(2\alpha+2n)\Gamma(\alpha+\frac n2+\frac12)}{\pi^n\Gamma(\alpha+1)\Gamma(2\alpha+n)\Gamma(\alpha+n+\frac 12)}
 \label{laplacetransform5}
\end{equation}
and $P(z)=z_1^2+\cdots+z_{n-1}^2-z_n^2$ satisfying $P(z)\subset \mathbb C\setminus(-\infty,0]$ for $z\in T_{\Lambda_n}.$

\begin{proof}
For $t=(t',t_n)\in \mathbb{R}^n$ and Let $a=4\pi t=(a',a_n)$, then based on the form of kernel in Theorem 1, we have
\begin{eqnarray*}
I(t)=\int_B e^{-4\pi y\cdot t}(\Delta(y))^{\alpha}dy=\int_0^{\infty}\int_{|y'|<y_n}e^{-a_ny_n-a'\cdot y'}(\Delta(y))^{\alpha}dy'dy_n.
\end{eqnarray*}
We now choose an orthogonal matrix $A$ such that $Ae_1=\frac{a'}{|a'|}$, $A^TA=I$ and $\{Ae_1,\ldots,Ae_{n-1}\}$ is also the orthogonal basis in $\mathbb R^{n-1}$, where $I$ is the identity matrix and $A^T$ is a transposed matrix of $A$. Hence $a'\cdot y'=Ae_1|a'|\cdot y'=|a'|y_1$. Write $y=(y',y_n)=(y_1,y'',y_n)$, then
\begin{eqnarray*}
I(t)&=&\int_0^{\infty}\int_{-y_n}^{y_n}\int_{|y''|<\sqrt{y_n^2-y_1^2}} e^{-a_ny_n-|a'|y_1}(y_n^2-y_1^2-|y''|^2)^{\alpha}dy''dy_1dy_n\\
&=&\int_0^{\infty}\int_{0}^{y_n}\int_{|y''|<\sqrt{y_n^2-y_1^2}} 2e^{-a_ny_n}\cosh (|a'|y_1)(y_n^2-y_1^2-|y''|^2)^{\alpha}dy''dy_1dy_n.\\
\end{eqnarray*}
where
$\cosh (s)=\frac{1}{2}(e^{s}+e^{-s}).
$
Letting $\sinh (s)=\frac{1}{2}(e^{s}-e^{-s})
$ and $y=\Phi(x)=(x_n\sinh x_1,x'',x_n\cosh x_1)$, together with $0\leq\sinh x_1<\cosh x_1$, we have $0<y_1<y_n$. Performing variables substitution to the above formula, then
\begin{eqnarray*}
I(t)=\int_0^{\infty}\int_{0}^{\infty}\int_{|x''|<x_n}2\cosh (|a'|x_n\sinh x_1)e^{-a_nx_n\cosh x_1}(x_n^2-|x''|^2)^{\alpha}(D\Phi)(x)dx''dx_1dx_n,
\end{eqnarray*}
in which $(D\Phi)(x)$ is the determinant of the Jacobian matrix of $\Phi(x)$ and $(D\Phi)(x)=x_n$. Put $t''=\frac{|x''|}{x_n}$, then
\begin{eqnarray*}
I(t)&=&\int_0^{\infty}\int_{0}^{\infty}\int_{|t''|<1}2\cosh (|a'|x_n\sinh x_1)e^{-a_nx_n\cosh x_1}(1-|t''|^2)^{\alpha}x_n^{2\alpha+n-1}dt''dx_1dx_n\\
&=& 2C_{\alpha}\int_0^{\infty}\int_{0}^{\infty}\cosh (|a'|x_n\sinh x_1)e^{-a_nx_n\cosh x_1}x_n^{2\alpha+n-1}dx_1dx_n,
\end{eqnarray*}
where $C_{\alpha}=\int_{|t''|<1}(1-|t''|^2)^{\alpha}dt''$. Now we compute the value of $C_{\alpha}$ as follows,
\begin{eqnarray*}
C_{\alpha}= \int_0^1\rho^{n-3}\int_{|\xi|=1}(1-\rho^2)^{\alpha}d\sigma_{n-2}(\xi)d\rho =S_{n-2}\int_0^1\rho^{n-3}(1-\rho^2)^{\alpha}d\rho,
\end{eqnarray*}
where $S_{n-2}=\frac{2\pi^{\frac{n-1}{2}}}{\Gamma(\frac{n-1}{2})}$ is the surface area of $n-2$ dimensional sphere. Letting $s=\rho^2$,
\begin{eqnarray*}
C_{\alpha}&=&\frac{S_{n-2}}{2}\int_0^1s^{\frac{n-3}{2}}(1-s)^{\alpha}s^{-\frac12}ds =\frac{ \pi^{\frac{n-1}{2}}}{\Gamma(\frac{n-1}{2})}B\left(\frac{n-2}{2},\alpha+1\right)\\
 &=&\frac{ \pi^{\frac{n-1}{2}}\Gamma(\frac n2-1)\Gamma(\alpha+1)}{\Gamma(\frac{n-1}{2})\Gamma(\frac n2+\alpha)}.
\end{eqnarray*}
Therefore,
\begin{eqnarray*}
I(t)&=& 2C_{\alpha}\int_0^{\infty}\int_{0}^{\infty}\cosh (|a'|x_n\sinh x_1)e^{-a_nx_n\cosh x_1}x_n^{2\alpha+n-1}dx_1dx_n\\
&=&C_{\alpha}\int_{0}^{\infty}\int_{-\infty}^{\infty}e^{-x_n\left(a_n\cosh x_1-|a'|\sinh x_1\right)}x_n^{2\alpha+n-1}dx_1dx_n.
\end{eqnarray*}
Let $s=-x_n\left(a_n\cosh x_1-|a'|\sinh x_1\right)$, then
\begin{eqnarray*}
I(t)&=&C_{\alpha}\int_{-\infty}^{\infty}\frac{1}{(a_n\cosh x_1-|a'|\sinh x_1)^{2\alpha+n}}\left(\int_{0}^{\infty}e^{s}s^{2\alpha+n-1}ds\right) dx_1\\
&=&C_{\alpha}\Gamma(2\alpha+n)\int_{-\infty}^{\infty}\frac{1}{\left(\frac{a_n\cosh x_1-|a'|\sinh x_1}{\sqrt{a_n^2-|a'|^2}}\right)^{2\alpha+n}(a_n^2-|a'|^2)^{\alpha+\frac n2}}dx_1.
\end{eqnarray*}
Since $\left(\frac{a_n}{\sqrt{a_n^2-|a'|^2}}\right)^2-\left(\frac{|a'|^2}{\sqrt{a_n^2-|a'|^2}}\right)^2=1$, put $\cosh t_0=\frac{a_n}{\sqrt{a_n^2-|a'|^2}}$ and $\sinh t_0=\frac{|a'|^2}{\sqrt{a_n^2-|a'|^2}}$, then
\begin{eqnarray*}
I(t)&=&\frac{C_{\alpha}\Gamma(2\alpha+n)}{(a_n^2-|a'|^2)^{\alpha+\frac n2}}\int_{-\infty}^{\infty}\frac{1}{\left(\cosh t_0\cosh x_1-\sinh t_0\sinh x_1\right)^{2\alpha+n}}dx_1\\
&=&\frac{C_{\alpha}\Gamma(2\alpha+n)}{(a_n^2-|a'|^2)^{\alpha+\frac n2}}\int_{-\infty}^{\infty}\frac{1}{\left(\cosh(x_1-t_0)\right)^{2\alpha+n}}dx_1\\
&=&\frac{C_{\alpha}\Gamma(2\alpha+n)}{(a_n^2-|a'|^2)^{\alpha+\frac n2}}\int_{-\infty}^{\infty}\frac{1}{\left(\cosh t\right)^{2\alpha+n}}dt.
\end{eqnarray*}
Write $J(t)=\int_{-\infty}^{\infty}\frac{1}{\left(\cosh t\right)^{2\alpha+n}}dt$ and let $(\cosh t)^2=s$,
\begin{eqnarray*}
J(t)=2\int_{-\infty}^{\infty}\frac{1}{2s^{\alpha+\frac n2+\frac12}(s-1)^{\frac12}}ds= \int_{-\infty}^{\infty}\frac{1}{ s^{\alpha+\frac n2+1}(1-\frac 1s)^{\frac12}}ds.
\end{eqnarray*}
Set $\frac 1s=k$, then
\begin{eqnarray*}
J(t)=\int_0^1 k^{\alpha+\frac n2+1}(1-k)^{-\frac12}k^{-2}dk=B\left(\alpha+\frac n2,\frac12\right)=\frac{\Gamma(\alpha+\frac n2)\Gamma({\frac12})}{\Gamma(\alpha+\frac{n+1}{2})}.
\end{eqnarray*}
As a result,
\begin{eqnarray*}
I(t)&=&\frac{\pi^{\frac{n-1}{2}}\Gamma(\frac n2-1)\Gamma(\alpha+1)\Gamma(2\alpha+n)\Gamma({\frac12})}{\Gamma(\frac{n-1}{2}) \Gamma(\alpha+\frac{n+1}{2})}(\Delta (a))^{-\alpha-\frac n2}\\
&=&\frac{\pi^{\frac{n-1}{2}}\Gamma(\frac n2-1)\Gamma(\alpha+1)\Gamma(2\alpha+n)\Gamma({\frac12})}{(4\pi)^{2\alpha+n}\Gamma(\frac{n-1}{2}) \Gamma(\alpha+\frac{n+1}{2})}(\Delta (t))^{-\alpha-\frac n2}\\
&=&\widetilde C_{\alpha,n}(\Delta (t))^{-\alpha-\frac n2},
\end{eqnarray*}
where $\widetilde C_{\alpha,n}=\frac{\pi^{\frac{n-1}{2}}\Gamma(\frac n2-1)\Gamma(\alpha+1)\Gamma(2\alpha+n)\Gamma({\frac12})}{(4\pi)^{2\alpha+n}\Gamma(\frac{n-1}{2}) \Gamma(\alpha+\frac{n+1}{2})}$. In order to obtain the representation form of the reproducing kernel $K_{\alpha,\Lambda_n}(z,w)$, $I(t)$ should be finite, i.e., $t_n^2>|t'|^2$. Therefore,
\begin{eqnarray*}
K(z)=\int_{\mathbb R^n}\frac{e^{2\pi it\cdot z}}{I(t)}dt=\frac{1}{\widetilde C_{\alpha,n}}\int_{\mathbb R^n}e^{2\pi it\cdot z}(\Delta (t))^{\alpha+\frac n2}dt.
\end{eqnarray*}
Let $K_1(z)=\widetilde C_{\alpha,n}K(z)$, then performing variables substitution as that in the computation of $I(t)$,
\begin{eqnarray*}
K_1(iy)&=&\int_{\mathbb R^n}e^{-2\pi t\cdot y}(\Delta (t))^{\alpha+\frac n2}dt\\
&=&\int_0^{\infty}\int_{|t'|<t_n}e^{-2\pi y_nt_n-2\pi t'\cdot y'}(\Delta (t))^{\alpha+\frac n2}dt'dt_n\\
&=&\int_0^{\infty}\int_{-t_n}^{t_n}\int_{|t''|<\sqrt{t_n^2-t_1^2}}e^{-2\pi y_nt_n-2\pi |y'|t_1}(t_n^2-|t_1|^2-|t''|^2)^{\alpha+\frac n2}dt''dt_1dt_n.
\end{eqnarray*}
Letting $t=\Phi (u)=(u_n\sinh u_1,u'',u_n\cosh u_1)$, then $(D\Phi)(u)=u_n$. Therefore,
\begin{equation*}
K_1(iy)=\int_0^{\infty}\int_{-\infty}^{\infty}\int_{|u''|<u_n}e^{-2\pi y_nu_n\cosh u_1-2\pi |y'|u_n\sinh u_1}(u_n^2-|u''|^2)^{\alpha+\frac n2}u_ndu''du_1du_n.
\end{equation*}
Put $s''=\frac{u''}{u_n}$, then
\begin{equation*}
K_1(iy)=\int_0^{\infty}\int_{-\infty}^{\infty}\int_{|s''|<1}e^{-2\pi y_nu_n\cosh u_1-2\pi |y'|u_n\sinh u_1}(1-|s''|^2)^{\alpha+\frac n2}u_n^{1+2\alpha+n+n-2}ds''du_1du_n,
\end{equation*}
in which
\begin{eqnarray*}
\int_{|s''|<1}(1-|s''|^2)^{\alpha+\frac n2}ds''
&=&\int_0^1\rho^{n-2-1}\int_{|\zeta|=1}(1-\rho^2)^{\alpha+\frac n2}d\sigma_{n-2}(\zeta)d\rho\\
&=&S_{D_{n-2}}\int_0^1\rho^{n-3}(1-\rho^2)^{\alpha+\frac n2}d\rho\\
&=&S_{D_{n-2}}\int_0^1s^{\frac{n-3}{2}}(1-s)^{\alpha+\frac n2}\frac12s^{-\frac12}ds=\frac{S_{D_{n-2}}}{2}B\left(\frac{n-2}{2},\alpha+\frac n2+1\right)
\end{eqnarray*}
by letting $s=\rho^2$. Here, $S_{D_{n-2}}$ is the surface area of $n-2$ dimensional unit ball and $B\left(\frac{n-2}{2},\alpha+\frac n2+1\right)$ is a beta function. Let $k_{\alpha,n}=\frac{S_{D_{n-2}}}{2}B\left(\frac{n-2}{2},\alpha+\frac n2+1\right)$. For $u_1\in\mathbb R$, $\sinh u_1=-\sinh(-u_1)$, we then have
\begin{eqnarray*}
K_1(iy)=k_{\alpha,n}\int_0^{\infty}\int_{-\infty}^{\infty} e^{-(2\pi y_nu_n\cosh u_1-2\pi |y'|u_n\sinh u_1)} u_n^{ 2\alpha+2n-1} du_1du_n.
\end{eqnarray*}
Let $2\pi y_nu_n\cosh u_1-2\pi |y'|u_n\sinh u_1=s$, then
\begin{eqnarray*}
K_1(iy)&=& k_{\alpha,n}\int_{-\infty}^{\infty} \frac{1}{(2\pi y_n \cosh u_1-2\pi |y'| \sinh u_1)^{2\alpha+2n}} \left(\int_0^{\infty}e^{-s} s^{ 2\alpha+2n -1}ds\right) du_1\\
&=&\frac{k_{\alpha,n}\Gamma(2\alpha+2n)}{(2\pi)^{2\alpha+2n}}\int_{-\infty}^{\infty} \frac{1}{( y_n \cosh u_1- |y'| \sinh u_1)^{2\alpha+2n}}du_1\\
&=&\frac{k_{\alpha,n}\Gamma(2\alpha+2n)}{(2\pi)^{2\alpha+2n}}\int_{-\infty}^{\infty} \frac{1}{\left( \frac{y_n \cosh u_1- |y'| \sinh u_1}{\sqrt{y_n^2-|y'|^2}}\right)^{2\alpha+2n}(y_n^2-|y'|^2)^{\alpha+n}}du_1.
\end{eqnarray*}
Put $\frac{y_n}{\sqrt{y_n^2-|y'|^2}}=\cosh t_0$ and $\frac{|y'|}{\sqrt{y_n^2-|y'|^2}}=\sinh t_0$, then
\begin{eqnarray*}
K_1(iy)&=&\frac{k_{\alpha,n}\Gamma(2\alpha+2n)}{(2\pi)^{2\alpha+2n}(y_n^2-|y'|^2)^{\alpha+n}}\int_{-\infty}^{\infty} \frac{1}{\left( \cosh t_0\cosh u_1-\sinh t_0\sinh u_1\right)^{2\alpha+2n}}du_1\\
&=&\frac{k_{\alpha,n}\Gamma(2\alpha+2n)}{(2\pi)^{2\alpha+2n}(y_n^2-|y'|^2)^{\alpha+n}}\int_{-\infty}^{\infty} \frac{1}{\left( \cosh (u_1-t_0)\right)^{2\alpha+2n}}du_1.
\end{eqnarray*}
Set $u_1-t_0=t$, $(\cosh t)^2=s$ and $\frac 1s=s'$, we have
\begin{eqnarray*}
K_1(iy)&=&\frac{k_{\alpha,n}\Gamma(2\alpha+2n)}{(2\pi)^{2\alpha+2n}(y_n^2-|y'|^2)^{\alpha+n}}\int_{-\infty}^{\infty} \frac{1}{\left( \cosh t\right)^{2\alpha+2n}}dt\\
&=&\frac{2k_{\alpha,n}\Gamma(2\alpha+2n)}{(2\pi)^{2\alpha+2n}(y_n^2-|y'|^2)^{\alpha+n}}\int_{0}^{\infty} \frac{1}{2 s^{\alpha+n}\sqrt{s-1}\sqrt{s}}ds\\
&=&\frac{ k_{\alpha,n}\Gamma(2\alpha+2n)}{(2\pi)^{2\alpha+2n}(y_n^2-|y'|^2)^{\alpha+n}}\int_{0}^{\infty} \frac{1}{ s^{\alpha+n+1}\sqrt{1-\frac 1s}}ds\\
&=&\frac{ k_{\alpha,n}\Gamma(2\alpha+2n)}{(2\pi)^{2\alpha+2n}(y_n^2-|y'|^2)^{\alpha+n}}\int_{0}^{1} (s')^{\alpha+n+1}(1-s')^{-\frac12}(s')^{-2}ds'\\
&=&\frac{ k_{\alpha,n}\Gamma(2\alpha+2n)}{(2\pi)^{2\alpha+2n}(y_n^2-|y'|^2)^{\alpha+n}} B\left(\alpha+n,\frac12\right).
\end{eqnarray*}
Therefore,
$$
K(iy)=\frac{4^{\alpha}\Gamma(\alpha+\frac n2+1)\Gamma(2\alpha+2n)\Gamma(\alpha+\frac n2+\frac12)}{\pi^n\Gamma(\alpha+1)\Gamma(2\alpha+n)\Gamma(\alpha+n+\frac 12)}(\Delta (y))^{-\alpha-n}.
$$
On the other hand, if we let $P(z)=z_1^2+\cdots+z_{n-1}^2-z_n^2$, then  $P(z)\in  \mathbb C\setminus(-\infty,0]$ for $z\in T_{\Lambda_n}$. Indeed,
\begin{eqnarray*}
P(z)&=&z_1^2+\cdots+z_{n-1}^2-z_n^2\\
&=&(x_1+iy_1)^2+\cdots+(x_{n-1}+iy_{n-1})^2-z_n^2\\
&=&(x_1^2-y_1^2)+\cdots+(x_{n-1}^2-y_{n-1}^2)-(x_n^2-y_n^2) +2i(x_1y_1+\cdots+x_{n-1}y_{n-1}-x_ny_n)\\
&=&\left(\sum_{k=1}^{n-1}x_k^2-x_n^2\right)+\left(y_n^2-\sum_{k=1}^{n-1}y_k^2\right) +2i\left(\sum_{k=1}^{n-1}x_ky_k-x_ny_n\right).
\end{eqnarray*}
It follows from $\sum_{k=1}^{n-1}x_ky_k=x_ny_n$ that $\sum_{k=1}^{n-1}x_k\frac{y_k}{y_n}=x_n$, then $x_n^2=(\sum_{k=1}^{n-1}x_k\frac{y_k}{y_n})^2$. H\"{o}lder's inequality implies that
\begin{equation*}
x_n^2\leq\left(\sum_{k=1}^{n-1}x_k^2\right)\left(\sum_{k=1}^{n-1}\frac{y_k^2}{y_n^2}\right)<\sum_{k=1}^{n-1}x_k^2.
\end{equation*}
Hence, $P(z)\in \mathbb C\setminus(-\infty,0]$ for $z\in T_B$. Therefore, $(P(z))^{-\alpha-n} $ is well defined for $z\in T_{\Lambda_n} $ and  $P(iy)=-y_1^2-y_2^2-\cdots-y_{n-1}^2+y_n^2=\Delta(y)$ for $y\in \Lambda_n$  and $P(iy)^{\alpha}=\rho(iy)$ is a weight function on $T_{\Lambda_n}$. Then
\begin{eqnarray*}
K(iy)=\frac{4^{\alpha}\Gamma(\alpha+\frac n2+1)\Gamma(2\alpha+2n)\Gamma(\alpha+\frac n2+\frac12)}{\pi^n\Gamma(\alpha+1)\Gamma(2\alpha+n)\Gamma(\alpha+n+\frac 12)}P(iy)^{-\alpha-n}.
\end{eqnarray*}
For $z=x+iy\in T_{\Lambda_n}$, $K(x+iy)$ admits Taylor expansion formula
\begin{equation*}
K(x+iy)=\sum_{k_1=0}^{\infty}\cdots\sum_{k_n=0}^{\infty}\frac{K^{(k_1+\cdots+ k_n)}(iy)}{k_1!\cdots k_n!}x_1^{k_1}\cdots x_n^{k_n},
\end{equation*}
then
\begin{equation*}
K(z)=\frac{4^{\alpha}\Gamma(\alpha+\frac n2+1)\Gamma(2\alpha+2n)\Gamma(\alpha+\frac n2+\frac12)}{\pi^n\Gamma(\alpha+1)\Gamma(2\alpha+n)\Gamma(\alpha+n+\frac 12)}P(z)^{-\alpha-n}.
\end{equation*}
As a result, for $ z,w\in T_{\Lambda_n}$,
\begin{equation*}
K_{\alpha,\Lambda_n}(z,w)=K(z-\bar{w})=C_{4,\alpha}P(z-\bar{w})^{-\alpha-n}.
\end{equation*}
where $C_{4,\alpha}$  is defined by  (\ref{laplacetransform5}). This proves (\ref{laplacetransform4}).
\end{proof}

\vskip 0.5cm{

\end{document}